\documentclass{amsart}
\usepackage{graphicx}
\newtheorem{thm}{Theorem}[section]
\newtheorem{cor}[thm]{Corollary}
\newtheorem{lem}[thm]{Lemma}

\theoremstyle{definition}

\theoremstyle{remark}
\newtheorem{rem}[thm]{Remark}
\numberwithin{equation}{section}

\begin{document}

\title[]{Solutions and stability of a variant of Van Vleck's and d'Alembert's functional equations}%
\author{E. Elqorachi , A. Redouani and Th M. Rassias}
\begin{abstract}
In this  paper. (1) We determine the complex-valued solutions of the
following variant of Van Vleck's functional equation
$$\int_{S}f(\sigma(y)xt)d\mu(t)-\int_{S}f(xyt)d\mu(t) = 2f(x)f(y),
\;x,y\in S,$$  where $S$ is a semigroup, $\sigma$ is an involutive
morphism of $S$, and $\mu$ is a complex measure that is linear
combinations of Dirac measures $(\delta_{z_{i}})_{i\in I}$, such
that for all $i\in I$, $z_{i}$ is  contained in the center of $S$.
(2) We  determine the  complex-valued continuous solutions of the
following variant of d'Alembert's functional equation
$$\int_{S}f(xty)d\upsilon(t)+\int_{S}f(\sigma(y)tx)d\upsilon(t) = 2f(x)f(y),
\;x,y\in S,$$ where $S$ is a topological semigroup, $\sigma$ is a
continuous involutive automorphism of $S$, and $\upsilon$ is a
complex measure with compact support and which is
$\sigma$-invariant. (3) We prove the superstability theorems of the
first functional equation.
\end{abstract}
\maketitle
\section{Introduction}
In his  two papers \cite{V1,V2}  Van Vleck  studied the continuous
solutions $f$ : $\mathbb{R} \longrightarrow \mathbb{R}$, $f\neq 0$
of the following functional equation
\begin{equation}\label{eq1}
f(x - y + z_0)-f(x + y + z_0) = 2f(x)f(y),\;x,y \in
\mathbb{R},\end{equation} where $z_0>0$ is fixed. He showed first
that all solutions are periodic with period $4z_0$, and then he
selected for his study any continuous solution with minimal period
$4z_0$. He proved that any such  solution has to be the sine
function
\\$f(x)=\sin(\frac{\pi}{2z_0}x)=\cos(\frac{\pi}{2z_0}(x-z_{0})$, $x\in \mathbb{R}$. \\
  Stetk\ae r [9,
Exercise 9.18] found the complex-valued solution of equation
\begin{equation}\label{eq3} f(xy^{-1}z_0)-f(xyz_0) =
2f(x)f(y),\;x,y \in G,\end{equation} on non
abelian groups $G$  and where  $z_0$ is a fixed element in the center of $G$.\\
Perkins and Sahoo \cite{P} replaced the group inversion by an
involution anti-automorphism $\tau$: $G\longrightarrow G$ and they
obtained the abelian, complex-valued solutions of equation
\begin{equation}\label{eqq4} f(x\tau(y)z_0)-f(xyz_0) = 2f(x)f(y),\;x,y
\in G.\end{equation}  Stetk\ae r \cite{St3} extends the results of
Perkins and Sahoo \cite{P} about equation (\ref{eqq4}) to the more
general case where $G$ is a semigroup and the solutions are not
assumed to be abelian.\\Recently,  Bouikhalene and Elqorachi
\cite{elq-boui} extends the results of Stetk\ae r's \cite{St3} and
obtain the solutions of the following extension of Van Vleck's
functional equations $$ \chi(y)f(x\tau(y)z_0)-f(xyz_0) =
2f(x)f(y),\;x,y \in S $$ and \begin{equation}\label{eq5}
\chi(y)f(\sigma(y)xz_0)-f(xyz_0) = 2f(x)f(y),\;x,y \in
M,\end{equation} where $S$ is a semigroup, $\chi$ is a
multiplicative, $M$ is a monoid, $\tau$ is an involution
anti-automorphism of  $S$ and $\sigma$ is an
involutive automorphism of $M$.\\
There has been quite a development of the theory of d'Alembert's
functional equation \begin{equation}\label{salma1}
    f(xy)+f(x\tau(y))=2f(x)f(y),\; x,y\in G,
\end{equation}  on non
abelian groups, as shown in works by Davison \cite{da1,da2} for
general groups, even monoids.  The non-zero solutions of equation
(\ref{salma1}) for general groups, even monoids are the normalized
traces of certain representations of the group $G$ on
$\mathbb{C}^{2}$ \cite{da1,da2}.\\ Stetk\ae r \cite{st4} obtained
the complex valued solutions of the following variant   of
d'Alembert's functional equation
\begin{equation}\label{salma3}
    f(xy)+f(\sigma(y)x)=2f(x)f(y),\; x,y\in S,
\end{equation} where  $\sigma$ is an involutive automorphism of the semigroup $S$.  The  solutions of equation (\ref{salma3}) are
of the form $f(x)=\frac{\chi(x)+\chi(\sigma(x))}{2}$, $x\in G$,
where $\chi$ is multiplicative.\\
In \cite{eb1} Ebanks and Stetk\ae r obtained the solutions $f,g$:
$G\longrightarrow \mathbb{C}$ of
 the following variant of Wilson's functional
equation (see also \cite{wilson5})
$$f(xy)+f(y^{-1}x)=2f(x)g(y),\; x,y\in G.
$$
In 1979, a type of stability was observed by J. Baker, J. Lawrence
and F. Zorzitto [5]. Indeed, they proved that if a function is
approximately exponential, then it is either a true exponential
function or bounded. Then the exponential functional equation is
said to be superstable. This result was the first result concerning
the superstability phenomenon of functional equations. Later, J.
Baker [4] (see also [1, 3, 6, 13]) generalized this result as
follows: Let $(S, .)$ be an arbitrary semigroup, and let $f:$
$S\longrightarrow \mathbb{C}$. Assume that $f$ is an approximately
exponential function, i.e., there exists a nonnegative number
$\delta$ such that $|f(x y)-f(x)f(y)|\leq\delta$ for all $x,y\in S$.
Then $f$ is either bounded or $f$ is a multiplicative function. The
result of Baker, Lawrence and Zorzitto [5] was generalized by L.
Sz\'ekelyhidi \cite{z1, ze3} in another way. We refer also to
\cite{omar0}, \cite{omar1}, \cite{omar2}, \cite{omar3}, \cite{J1},
\cite{J2}, \cite{K2}, \cite{omar5}, \cite{par} and \cite{th} for
other results concerning the stability and the superstability of
functional equations
\\In the  first part of this paper we extend the above
results to the following generalization of Van Vleck's functional
equation for the sine
\begin{equation}\label{eq555}
\int_{S}f(\sigma(y)xt)d\mu(t)-\int_{S}f(xyt)d\mu(t) = 2f(x)f(y),\;
x,y\in S,\end{equation}
 where $S$ is  a semigroup and  $\sigma$ is an involutive morphism : That
 is $\sigma$ is an involutive automorphism:
$\sigma(xy)=\sigma(x)\sigma(y)$ and $\sigma(\sigma(x))=x$ for all
$x,y$ or $\sigma$ is an involutive anti-automorphism:
$\sigma(xy)=\sigma(y)\sigma(x)$ and $\sigma(\sigma(x))=x$ for all
$x,y$, and $\mu$ is assumed to be a complex measure that is linear
combination of Dirac measures $(\delta_{z_{i}})_{i\in I}$, with
$z_{i}$ contained in the center of $S$, for all $i\in I$. \\The main
idea is to relate the functional equation (\ref{eq555}) to  to the
following variant of d'Alembert's functional equation
\begin{equation}\label{eq666}
    g(xy)+g(\sigma(y)x)=2g(x)g(y),\;x,y\in S
\end{equation} and we apply the result obtained by Stetk\ae r \cite{ st4, wilson5}.\\
In section 3, we obtain the the  complex-valued continuous solutions
of the following variant of d'Alembert's functional equation
\begin{equation}\label{elqorachi2}
\int_{S}f(xty)d\upsilon(t)+\int_{S}f(\sigma(y)tx)d\upsilon(t) =
2f(x)f(y), \;x,y\in S,\end{equation}
 where $S$ is a topological semigroup,
$\sigma$ is a continuous involutive automorphism of $S$, and
$\upsilon$ is a complex measure with compact support and which is
$\sigma$-invariant. That is
$\int_{S}h(\sigma(t))d\upsilon(t)=\int_{S}h(t)d\upsilon(t)$ for all
continuous function $h$ on $S$.\\In the last section we obtain the
superstability theorems of the functional equation (\ref{eq555}) .\\
In all proofs of the results of this paper we use without explicit
mentioning the assumption that   $z_i$ is contained in the center of
$S$ for all $i\in I$  and its consequence $\sigma(z_i)$ is contained
in the center of $S$.
\section{The complex-valued solutions of equation (\ref{eq555}) on semigroups}
In this section we determine the solutions of the variant Van
Vleck's functional equation (\ref{eq555}) on semigroups. We first
prove the following useful lemmas.
\begin{lem} Let $S$ be semigroup, let $\sigma$ :
$S\longrightarrow S$ be an involutive morphism of $S$ and   $\mu$ be
a complex measure that is linear combination of Dirac measures
$(\delta_{z_{i}})_{i\in I}$, with $z_{i}$  contained in the center
of $S$ for all $i\in I$. \\Let $f$: $S\longrightarrow \mathbb{C}$ be
a non-zero solution of equation (\ref{eq555}). Then   for all
$x,y\in S$ we have
\begin{equation}\label{eqo7}
    f(x)=-f(\sigma(x)),
\end{equation}
\begin{equation}\label{eqo8}
    \int_{S}f(t)d\mu(t)\neq 0,
\end{equation}
\begin{equation}\label{eqoo8}
    f(\sigma(y)x)=-f(\sigma(x)y),
\end{equation}

\begin{equation}\label{eqo9}
     \int_{S} \int_{S}f(x\sigma(t)s)d\mu(t)d\mu(s)=f(x) \int_{S} f(t)d\mu(t),
\end{equation}
\begin{equation}\label{eqo10}
     \int_{S} \int_{S}f(xts)d\mu(t)d\mu(s)=-f(x)\int_{S} f(t)d\mu(t),
\end{equation}
\begin{equation}\label{eqo11}
     \int_{S} f(\sigma(x)t)d\mu(t)= \int_{S}f(xt)d\mu(t),
\end{equation}
\begin{equation}\label{elq10}
 \int_{S}f(x\sigma(t))d\mu(t))=\int_{S}f(\sigma(x)\sigma(t))d\mu(t),\end{equation}
 \begin{equation}\label{omar100}
     \int_{S} \int_{S}f(ts)d\mu(t)d\mu(s)=\int_{S}
     \int_{S}f(t\sigma(s))d\mu(t)d\mu(s)=0.
\end{equation}
 \end{lem}\begin{proof} Equation (2.2): Let $f\neq 0$  be a non-zero solution of equation (\ref{eq555}) and
assume that\\
 $\int_{S}f(t)d\mu(t)=0$. Taking $y=s$ in
 equation (\ref{eq555}) and integrate the result obtained with respect to $s$ we get
\begin{equation}\label{omar1}
 \int_{S}\int_{S}f(\sigma(s)xt)d\mu(s)d\mu(t)-\int_{S}\int_{S}f(xst)d\mu(s)d\mu(t) =
 2f(x)\int_{S}f(s)d\mu(s)=0\end{equation}
Replacing  $y$ by $ys$ in (\ref{eq555}) and integrating the result
 with respect to $s$ and using (\ref{omar1}) and (\ref{eq555}) we get
 $$\int_{S}\int_{S}f(\sigma(y)xt\sigma(s))d\mu(t)d\mu(s)-\int_{S}\int_{S}f(xyst)d\mu(s)d\mu(t) = 2f(x)\int_{S}f(ys)d\mu(s)$$
 $$=\int_{S}\int_{S}f(\sigma(y)xts)d\mu(t)d\mu(s)-\int_{S}\int_{S}f(xyst)d\mu(s)d\mu(t) $$
 $$= 2f(y)\int_{S}f(xt)d\mu(t),$$
 which implies that $f(y) \int_{S}f(xs)d\mu(s)=f(x)
 \int_{S}f(ys)d\mu(s)$ for all $x,y\in S.$ Since $f\neq 0$, then
 there exists $\alpha\in \mathbb{C}\backslash \{0\}$ such that
  $\int_{S}f(xs)d\mu(s)=-\alpha
 f(x)$ for all $x\in S.$ Substituting this into (\ref{eq555}) we get
 \begin{equation}\label{elq1}
   f(xy)-f(\sigma(y)x)=2\frac{f(x)}{\alpha}f(y)\; \text{for all }\;x,y \in S.
 \end{equation}
By interchanging $x$ with $y$ in (\ref{elq1}) we get
\begin{equation}\label{mohhh1}
 f(yx)-f(\sigma(x)y)=\frac{2}{\alpha}f(x)f(y)
\end{equation} If we replace $y$ by $\sigma(y)$ in (\ref{elq1}) we
have \begin{equation}\label{mohhh2}
 f(x\sigma(y))-f(yx)=\frac{2}{\alpha}f(x)f(\sigma(y))
\end{equation} By adding    (\ref{mohhh2}) and (\ref{mohhh1})  we obtain
\begin{equation}\label{mohhh3}
 f(x\sigma(y))-f(\sigma(x)y)=\frac{2}{\alpha}f(x)[f(\sigma(y))+f(\sigma(y))]
\end{equation} By replacing $x$ by $\sigma(x)$ in (\ref{mohhh3}) we
get \begin{equation}\label{mohhh4}
 -f(xy)+f(\sigma(x)\sigma(y))=\frac{2}{\alpha}f(\sigma(x))[f(\sigma(y))+f(\sigma(y))]
\end{equation} If we replace $y$ by $\sigma(y)$ in (\ref{mohhh3}) we
get \begin{equation}\label{mohhh5}
 -f(\sigma(x)\sigma(y))+f(xy)=\frac{2}{\alpha}f(x)[f(\sigma(y))+f(\sigma(y))].
\end{equation}Now, by adding   (\ref{mohhh4}) and
(\ref{mohhh5})  we have
\begin{equation}\label{mohhh6}
 [f(x)+f(\sigma(x))][f(\sigma(y))+f(\sigma(y))]=0
\end{equation} That is $f(\sigma(x))=-f(x)$ for all $x\in S$. Now, we will discuss two cases.
Case 1: If $\sigma$ is an involutive anti-automorphism. Then from
$f(\sigma(x))=-f(x)$ for all $x\in S$ we have
$f(\sigma(y)x)=-f(\sigma(x)y)$ for all $x,y\in S$ and equation
 (\ref{elq1}) can be written as follows
 \begin{equation}\label{elq2}
   f(xy)+f(\sigma(x)y)=2\frac{f(x)}{\alpha}f(y)\; \text{for all }\;x,y \in S.
 \end{equation}The left hand side of (\ref{elq2}) is unchanged under
 interchange of $x$ and $\sigma(x)$, so we get $f(x)=f(\sigma(x))$
 for all $x\in S$. Now,
 $f(x)=-f(\sigma(x))=-f(x)$ implies that $f=0$. This contradicts the
assumption that $f\neq 0$.\\Case 2:  If $\sigma$ is an involutive
automorphism. Then from $f(\sigma(x))=-f(x)$ for all $x\in S$ we
have $f(\sigma(y)x)=-f(y\sigma(x))$ for all $x,y\in S$ and equation
 (\ref{elq1}) can be written as follows
 \begin{equation}\label{elq2}
   f(xy)+f(y\sigma(x))=2\frac{f(x)}{\alpha}f(y)\; \text{for all }\;x,y \in S.
 \end{equation}By replacing $x$ by  $\sigma(x)$ in (\ref{elq2}) and using $f(\sigma(x))=-f(x)$
 we get $$l(yx)+l(\sigma(x)y)=2l(x)l(y)\; \text{for all }\;x,y \in S.$$
 where $l=-\frac{f}{\alpha}$. So, from \cite{st4} $f(\sigma(x))=f(x)$ for all $x\in
 S$. Consequently, $f(\sigma(x))=f(x)=-f(x)$ for all $x\in S$, which implies that $f=0$. This contradict
  the assumption that $f\neq0$ and  this proves the
assertion (\ref{eqo8}).\\ Equation (2.3): Replacing $y$ by $yt$ in
(\ref{eq555}) and integrating the result obtained with respect to
$t$ we get
\begin{equation}\label{elq4}
\int_{S}\int_{S}f(\sigma(y)xs\sigma(t))d\mu(t)d\mu(s)-\int_{S}\int_{S}f(xyst)d\mu(s)d\mu(t)=2f(x)\int_{S}f(yt)d\mu(t).\end{equation}
Replacing  $x$ by $xs$ in (\ref{eq555}) and integrating the result
obtained with respect to $s$ we get
\begin{equation}\label{elq4}
\int_{S}\int_{S}f(\sigma(y)xst)d\mu(t)d\mu(s)-\int_{S}\int_{S}f(xyst)d\mu(s)d\mu(t)=2f(y)\int_{S}f(xs)d\mu(s).\end{equation}
Subtracting these equations results in
\begin{equation}\label{elq5}
\int_{S}\int_{S}f(\sigma(y)xs\sigma(t))d\mu(t)d\mu(s)-\int_{S}\int_{S}f(\sigma(y)xst)d\mu(t)d\mu(s)
\end{equation}$$=2f(x)\int_{S}f(yt)d\mu(t)-2f(y)\int_{S}f(xs)d\mu(s).$$
Since from (\ref{eq555}) we have
$$\int_{S}\int_{S}f(\sigma(y)xs\sigma(t))d\mu(t)d\mu(s)-\int_{S}\int_{S}f(\sigma(y)xst)d\mu(t)d\mu(s)$$
$$=\int_{S}[\int_{S}f(\sigma(t)\sigma(y)xs)d\mu(s)-\int_{S}f(\sigma(y)xts)d\mu(s)]d\mu(t)$$
$$=\int_{S}2f(t)f(\sigma(y)x)d\mu(t)=2f(\sigma(y)x)\int_{S}f(t)d\mu(t).$$
Then we get
\begin{equation}\label{elq6}
f(x)\int_{S}f(yt)d\mu(t)-f(y)\int_{S}f(xs)d\mu(s)=f(\sigma(y)x)\int_{S}f(t)d\mu(t)\end{equation}
for all $x,y\in S.$ Since
$f(x)\int_{S}f(yt)d\mu(t)-f(y)\int_{S}f(xs)d\mu(s)=-[f(y)\int_{S}f(xt)d\mu(t)-f(x)\int_{S}f(ys)d\mu(s)]$,
then we obtain
$f(\sigma(y)x)\int_{S}f(t)d\mu(t)=-f(\sigma(x)y)\int_{S}f(t)d\mu(t)$.
Now, by using (\ref{eqo8}) we deduce (\ref{eqoo8}).\\
Equation (2.7): Replacing $x$ by $x\sigma(s)$ in (\ref{eq555}) and
integrating the result obtained with respect to $s$ we get
\begin{equation}\label{elq7}
  \int_{S}\int_{S}f(\sigma(y)xs\sigma(s)t)d\mu(s)d\mu(t)-\int_{S}\int_{S}f(xy\sigma(s)t)d\mu(s)d\mu(t)\end{equation}$$=2f(y)\int_{S}f(x\sigma(s))d\mu(s).
$$ By using (\ref{eqoo8}) we have
$f(x(y\sigma(s)t))=-f(\sigma(y\sigma(s)t)\sigma(x))=-f(\sigma(y)\sigma(x)\sigma(t)s)$
and then
$$\int_{S}\int_{S}f(xy\sigma(s)t)d\mu(s)d\mu(t)=-\int_{S}\int_{S}f(\sigma(y)\sigma(x)\sigma(s)t)d\mu(s)d\mu(t),$$
so equation (\ref{elq7}) can be written as follows
\begin{equation}\label{elq8}
  \int_{S}\int_{S}f(\sigma(y)xs\sigma(s)t)d\mu(s)d\mu(t)+\int_{S}\int_{S}f(\sigma(y)\sigma(x)\sigma(s)t)d\mu(s)d\mu(t)\end{equation}
  $$=2f(y)\int_{S}f(x\sigma(s)d\mu(s).
$$ The left hand side of (\ref{elq8}) is unchanged under
 interchange of $x$ and $\sigma(x)$, so since $f\neq0$ we get
 (\ref{elq10}). By using (\ref{eqoo8}) and (\ref{elq10})  we get
 $$\int_{S}f(xs)d\mu(s)=-\int_{S}f(\sigma(s)\sigma(x))d\mu(s)$$
 $$=-\int_{S}f(x\sigma(s))d\mu(s) =\int_{S}f(\sigma(x)s)d\mu(s).$$
 This proves (\ref{eqo11}).
 \\Equation (2.1): By replacing $x$ by $\sigma(x)$ in (\ref{eq555}) we
 obtain
\begin{equation}\label{omar10}
 \int_{S}f(\sigma(y)\sigma(x)t)d\mu(t)-\int_{S}f(\sigma(x)yt)d\mu(t)
 =2f(\sigma(x))f(y)\end{equation}
  If  $\sigma$ is an involutive automorphism then  from (\ref{eqo11}) we have
   $$\int_{S}f(\sigma(y)\sigma(x)t)d\mu(t)= \int_{S}f(\sigma(yx)t)d\mu(t)=\int_{S}f(yxt)d\mu(t)$$
   and it follows from (2.25) that
 $$\int_{S}f(yxt)d\mu(t)-\int_{S}f(\sigma(x)yt)d\mu(t)=2f(\sigma(x))f(y).$$
 Since
 $$\int_{S}f(yxt)d\mu(t)-\int_{S}f(\sigma(x)yt)d\mu(t)=-\bigg[\int_{S}f(\sigma(x)yt)d\mu(t)-\int_{S}f(yxt)d\mu(t)\bigg]=-2f(y)f(x),$$
  then we conclude that $$-2f(x)f(y)=2f(\sigma(x))f(y)$$ for all
 $x,y\in S$. Since $f\neq0$ then we get (\ref{eqo7}).\\ If   $\sigma$ is an involutive anti-automorphism
 then from (\ref{eqo11}) we have $$\int_{S}f(\sigma(y)\sigma(x)t)d\mu(t)=
  \int_{S}f(\sigma(xy)t)d\mu(t)=\int_{S}f(xyt)d\mu(t)$$ and
  $\int_{S}f(\sigma(x)yt)d\mu(t)=\int_{S}f(\sigma(y)xt)d\mu(t)$.
  Equation (\ref{omar10}) implies that
  $$\int_{S}f(xyt)d\mu(t)-\int_{S}f(\sigma(x)yt)d\mu(t)=2f(\sigma(x))f(y)=-\bigg[\int_{S}f(\sigma(x)yt)d\mu(t)-\int_{S}f(xyt)d\mu(t)\bigg]$$$$
  =-\bigg[\int_{S}f(\sigma(y)xt)d\mu(t)-\int_{S}f(xyt)d\mu(t)\bigg]=-2f(x)f(y).$$ Since $f\neq0$ then we get again
  (\ref{eqo7}).
 \\
 Equation (2.4): Putting $x=\sigma(s)$ in (\ref{eq555}), using (\ref{eqo7}) and integrating the result obtained with respect to $s$ we get by a computation that
  $$\int_{S}\int_{S}f(\sigma(y)\sigma(s)t)d\mu(s)d\mu(t)-\int_{S}\int_{S}f(\sigma(s)yt)d\mu(s)d\mu(t) =2f(y)\int_{S}f(\sigma(s))d\mu(s)$$$$=-2f(y)\int_{S}f(s)d\mu(s).$$
  Since
 $$\int_{S}\int_{S}f(\sigma(y)\sigma(s)t)d\mu(s)d\mu(t)=-\int_{S}\int_{S}f(ys\sigma(t))d\mu(s)d\mu(t),$$ then we get
 $$\int_{S}\int_{S}f(\sigma(s)yt)d\mu(s)d\mu(t)=f(y)\int_{S}f(s)d\mu(s)$$  for all $y\in S$, which proves (\ref{eqo9}).\\Equation (2.5):
 By using (\ref{eqo9}), replacing $y$ by $s$ in (\ref{eq555})  and
integrating the result obtained with respect to $s$ we
 get $$\int_{S}\int_{S}f(\sigma(s)xt)d\mu(s)d\mu(t)-\int_{S}\int_{S}f(xst)d\mu(s)d\mu(t)$$$$
 =2f(x)\int_{S}f(s)d\mu(s)=f(x)\int_{S}f(s)d\mu(s)-\int_{S}\int_{S}f(xst)d\mu(s)d\mu(t).$$
 Equation (2.8): By replacing $x$ by $s$ in (\ref{eqo11}) and integrating the result obtained with respect to $s$ we get
  $\int_{S}\int_{S} f(\sigma(s)t)d\mu(s)d\mu(t)=
  \int_{S}\int_{S}f(st)d\mu(s)d\mu(t).$ From (\ref{eqoo8}) we have
  $f(\sigma(s)t)=-f(\sigma(t)s)$ for all $s,t\in S$, then $$\int_{S}\int_{S} f(\sigma(s)t)d\mu(s)d\mu(t)=-\int_{S}\int_{S}
  f(\sigma(s)t)d\mu(s)d\mu(t).$$
 Which  implies  (\ref{omar100}) and this completes the proof.\end{proof}
\begin{lem} Let $f$: $S\longrightarrow \mathbb{C}$ be a non-zero solution of equation (\ref{eq555}).
 Then\\
 (1) The function defined by $$g(x)\;:=\frac{\int_{S}f(xt)d\mu(t)}{\int_{S}f(t)d\mu(t)}\;
\text{for} \;x\in S$$ is a non-zero  solution of the variant of
d'Alembert's functional equation (\ref{salma3}). Furthermore, $
\int_{S}g(s)d\mu(s)=0.$ (2) The function $g$ from (1) has the form
$g=\frac{\chi+\chi\circ \sigma}{2}$, where $\chi$ :
$S\longrightarrow \mathbb{C}$, $\chi\neq 0$, is a multiplicative
function.\end{lem}
\begin{proof}
 (1)  From (\ref{eqo9}), (\ref{eqo10}), (\ref{eq555}) and the definition of $g$ we have $$\bigg(\int_{S}f(t)d\mu(t)\bigg)^{2}[g(xy)+g(\sigma(y)x)]=\int_{S}f(t)d\mu(t) \int_{S}f(\sigma(y)xs)d\mu(s)+\int_{S}f(t)d\mu(t) \int_{S}f(xys)d\mu(s)$$
 $$=\int_{S}\int_{S}\int_{S}f(\sigma(y)xs\sigma(t)k)d\mu(s)d\mu(t)d\mu(k)-\int_{S}\int_{S}\int_{S}f(xystk)d\mu(s)d\mu(t)d\mu(k)$$$$=
  \int_{S}\int_{S}\int_{S}f(\sigma(yt)(xs)k)d\mu(s)d\mu(t)d\mu(k)-\int_{S}\int_{S}\int_{S}f((xs)(yt)k)d\mu(s)d\mu(t)d\mu(k)$$
  $$=2\int_{S}f(xs)d\mu(s)\int_{S}f(yt)d\mu(t).$$
 Dividing by $(\int_{S}f(t)d\mu(t))^{2}$ we get $g$ satisfies the variant of d'Alembert's functional equation (\ref{salma3}).\\
 From (\ref{eqo10}) and the definition of $g$ we get  $$\int_{S}\int_{S}g(ts)d\mu(t)d\mu(s)=\frac{\int_{S}\int_{S}\int_{S}f(s'ts)d\mu(t)d\mu(s)d\mu(s')}{\int_{S}f(s)d\mu(s)}$$
 $$=\frac{-\int_{S}f(s')d\mu(s')\int_{S}f(s)d\mu(s)}{\int_{S}f(s)d\mu(s)}=-\int_{S}f(s)d\mu(s)\neq 0.$$
 From (\ref{omar100}) and the definition of $g$ we get $$\int_{S}g(s)d\mu(s)=\frac{\int_{S}f(st)d\mu(s)d\mu(t)
 }{\int_{S}f(s)d\mu(s)}=\frac{0
 }{\int_{S}f(s)d\mu(s)}=0$$
Furthermore, $\int_{S}\int_{S}g(st)d\mu(t)d\mu(s)\neq 0$, so  $g$ is
non-zero   solution of equation (\ref{salma3}).  \\(2) Let $f$ be a
non-zero solution of (\ref{eq555}). Replacing $x$ by $xs$ in
(\ref{eq555}) and integrating the result obtained with respect to
$s$ we get
\begin{equation}\label{haj1}
\int_{S}\int_{S}f(\sigma(y)xst)d\mu(s)d\mu(t)-\int_{S}\int_{S}f(xyst)d\mu(s)d\mu(t)=2f(y)\int_{S}f(xs)d\mu(s).\end{equation}
By using  (\ref{eqo9}), (\ref{eqo10}) equation (\ref{haj1}) can be
written as follows
\begin{equation}\label{haj2}
-f(\sigma(y)x)+f(xy)=2f(y)g(x),\;x,y\in S,\end{equation} where $g$
is the function defined above. If we replace $y$ by $ys$ in
(\ref{eq555}) and integrating the result obtained with respect to
$s$ we get
\begin{equation}\label{haj3}
\int_{S}\int_{S}f(\sigma(y)x\sigma(s)t)d\mu(s)d\mu(t)-\int_{S}\int_{S}f(xyst)d\mu(s)d\mu(t)=2f(x)\int_{S}f(ys)d\mu(s).\end{equation}
By using  (\ref{eqo9}), (\ref{eqo10}) we obtain
\begin{equation}\label{haj4}
f(\sigma(y)x)+f(xy)=2f(x)g(y),\;x,y\in S.\end{equation} By adding
(\ref{haj4}) and  (\ref{haj2}) we get that the pair$f,g$ satisfies
the sine addition law $$f(xy)=f(x)g(y)+f(y)g(x)\; \text{for all }\;
x,y\in S.$$ Now, in view of [\cite{ebanks}, Lemma 3.4.] $g$ is
abelian. Since $g$ is a non-zero solution of d'Alembert's functional
equation (\ref{salma3}) then  from [\cite{07}, Theorem 9.21] there
exists a non-zero multiplicative function $\chi$: $S\longrightarrow
\mathbb{C}$ such that $g=\frac{\chi+\chi\circ\sigma}{2}$. This
completes the proof.
\end{proof}
 Now we are ready to prove the  main result of the present section.
 \begin{thm} The non-zero solutions $f$ : $S\longrightarrow \mathbb{C}$ of
the functional equation (\ref{eq555}) are the functions of the form
\begin{equation}\label{elq300}
    f=\frac{\chi\circ\sigma-\chi}{2}\int_{S}\chi(t)d\mu(t),
\end{equation} where $\chi$ : $S\longrightarrow \mathbb{C}$ is a
multiplicative function such that $\int_{S}\chi(t)d\mu(t)\neq 0$ and
$\int_{S}\chi(\sigma(t))d\mu(t)=-\int_{S}\chi(t)d\mu(t)$.
\\If $S$ is a
topological semigroup and that $\sigma$ : $S\longrightarrow S$ is
continuous then the non-zero solution $f$ of equation (\ref{eq555})
is continuous, if and only if $\chi$ is continuous.
\end{thm}\begin{proof} Simple computations show that $f$ defined by (\ref{elq300}) is a solution of (\ref{eq555}). Conversely,  let $f$ : $S\longrightarrow \mathbb{C}$ be a
non-zero solution of the functional equation (\ref{eq555}). Putting
$y=s$ in (\ref{eq555}) and integrating the result obtained with
respect to $s$ we get
\begin{equation}\label{elqorachi1222}
f(x)=\frac{\int_{S}\int_{S}f(\sigma(s)xt)d\mu(s)d\mu(t)-\int_{S}\int_{S}\int_{S}f(xst)d\mu(s)d\mu(t)}{2\int_{S}f(s)d\mu(s)}\end{equation}
$$=\frac{1}{2}( \int_{S}g(\sigma(s)x)d\mu(s)-\int_{S}g(xs)d\mu(s)),$$
 where g is the function defined  by $g=\frac{\chi+\chi\circ \sigma}{2}$,
and where $\chi$ : $S\longrightarrow \mathbb{C}$, $\chi\neq 0$ is a
multiplicative function. Substituting this into
(\ref{elqorachi1222}) we find that $f$ has the form
\begin{equation}\label{equo1àà}
    f=\frac{\int_{S}\chi(s)d\mu(s)-\int_{S}\chi(\sigma(s))d\mu(s)}{2}\frac{ \chi\circ\sigma-\chi}{2}.
\end{equation}Furthermore, from (\ref{eqo11})  $f$ satisfies
$ \int_{S}f(\sigma(x)s)d\mu(s)=\int_{S}f(xs)d\mu(s)$ for all $x\in
S$. By applying the last expression of $f$ in  (\ref{eqo11}) we get
after computations that
$$\bigg[\int_{S}\chi(\sigma(t))d\mu(t)+\int_{S}\chi(t)d\mu(t)\bigg][\chi-\chi\circ\sigma]=0.$$ Since $\chi\neq   \chi\circ\sigma$, we
obtain
 $\int_{S}\chi(\sigma(t))d\mu(t)+\int_{S}\chi(t)d\mu(t)=0$ and  then from (\ref{equo1àà})
 we have
$$f=\frac{ \chi\circ\sigma-\chi}{2}\int_{S}\chi(t)d\mu(t).$$
For the topological statement we use [\cite{07}, Theorem 3.18(d)].
This completes the proof. \end{proof} If  $\mu=\delta_{z_0}$ where
$z_0$ is a fixed element of the center of $S$ we get the following
particular
 result obtained by Bouikhalene and Elqorachi on monoids \cite{elq-boui}.
\begin{cor}  \cite{elq-boui} Let $M$ be a monoid, let $\sigma$: $M\longrightarrow M$ be an involutive automorphism. The non-zero solutions $f$ : $S\longrightarrow \mathbb{C}$ of
the functional equation
$$f(\sigma(y)xz_0)-f(xyz_0)=2f(x)f(y),\; x,y\in M$$ are the functions of
the form
\begin{equation} f=\chi(z_0)\frac{\chi\circ\sigma-\chi}{2},
\end{equation}where $\chi$ : $M\longrightarrow \mathbb{C}$ is a
multiplicative function such that $\chi(z_0)\neq 0$ \\and
$\chi(\sigma(z_0))=-\chi(z_0)$.
\end{cor} \section{The complex-valued  continuous solutions of equation  (\ref{elqorachi2}) on a topological semigroup.}
The following lemma will be used to construct some particular
solutions of equation (\ref{elqorachi2}). \begin{lem} Let $S$ be a
locally compact semigroup, let  $\upsilon$ be  a complex measure
with compact support. The continuous solutions of the functional
equation
\begin{equation}\label{elqorachi3}
\int_{S}f(xyt)d\upsilon(t)=f(x)f(y),\;x,y\in S
\end{equation} are the functions
\begin{equation}\label{elqorachi4}
    f=\chi\int_{S}\chi(t)d\upsilon(t),
\end{equation} where $\chi$: $S\longrightarrow \mathbb{C}$ is a
continuous  multiplicative function on $S$.\end{lem}
\begin{proof}
We use in the proof similar Stetk\ae r's computations [\cite{stkan},
Proposition 16] used for the special case of
$\upsilon=\delta_{s_{0}}$ , where $s_{0}\in S$. Let $f$ be a
continuous solution of (\ref{elqorachi3}). Replacing $y$ by $s$ in
(\ref{elqorachi3}) and integrating the result obtained with respect
to $s$ we get
\begin{equation}\label{elqorachi5}
   \int_{S} \int_{S}f(xst)d\upsilon(s)d\upsilon(t)=f(x)\int_{S}f(s)d\upsilon(s),\;x\in
    S.
\end{equation}
Assume that $f\neq 0$, we will show that
$\int_{S}f(s)d\upsilon(s)\neq 0$. By replacing $x$ by $xs$,  $y$ by
$yk$ in (\ref{elqorachi3}) and integrating the result obtained with
respect to $s$ and $k$ we get
$\int_{S}\int_{S}\int_{S}f(xsykt)d\upsilon(s)d\upsilon(k)d\upsilon(t)=\int_{S}f(xs)d\upsilon(s)\int_{S}f(yk)d\upsilon(k)$.
On the other hand, from (\ref{elqorachi5}) we have
$$\int_{S}\int_{S}\int_{S}f(xsykt)d\upsilon(s)d\upsilon(k)d\upsilon(t)=\int_{S}f(xsy)d\upsilon(s)\int_{S}f(t)d\upsilon(t).$$
So, if $\int_{S}f(s)d\upsilon(s)=0$ then we get
$\int_{S}f(xs)d\upsilon(s)\int_{S}f(yk)d\upsilon(k)=0$ for all
$x,y\in S$ and it follows that
$\int_{S}f(xys)d\upsilon(s)=0=f(x)f(y) $ for all $x,y\in S$. Which
contradicts the assumption that $f\neq 0$.\\ From (\ref{elqorachi3})
and (3.3) we have
$$\int_{S}f(s)d\upsilon(s)\int_{S}f(xty)d\upsilon(t)=\int_{S}\int_{S}f(xtysk)d\upsilon(t)d\upsilon(s)d\upsilon(k)$$
$$=f(x)\int_{S}\int_{S}f(tys)d\upsilon(t)d\upsilon(s)=f(x)f(y)\int_{S}f(t)d\upsilon(t).$$
This implies that
\begin{equation}\label{elqorachi6}
\int_{S}f(xty)d\upsilon(t)=f(x)f(y)
\end{equation}for all $x,y\in S.$ Now, let
$$\chi(x)=\frac{\int_{S}f(xt)d\upsilon(t)}{\int_{S}f(t)d\upsilon(t)},\; x\in
S.$$ In view of  (\ref{elqorachi3}), (3.3) and (\ref{elqorachi6}) we
have
$$\bigg(\int_{S}f(t)d\upsilon(t)\bigg)^{2}\chi(x)\chi(y)=\int_{S}f(xt)d\upsilon(t)\int_{S}f(ys)d\upsilon(s)$$
$$=\int_{S}\int_{S}[f(xt)f(ys)]d\upsilon(t)d\upsilon(s)=\int_{S}\int_{S}[\int_{S}f(xtysk)d\upsilon(k)]d\upsilon(t)d\upsilon(s)$$
$$=\int_{S}f(k)d\upsilon(k)\int_{S}f(xty)d\upsilon(t)=\int_{S}f(k)d\upsilon(k)\int_{S}f(xyt)d\upsilon(t)$$
$$=\bigg(\int_{S}f(k)d\upsilon(k)\bigg)^{2}\frac{\int_{S}f(xyt)d\upsilon(t)}{\int_{S}f(t)d\upsilon(t)}=\bigg(\int_{S}f(k)d\upsilon(k)\bigg)^{2}\chi(xy).$$
 Which proves that $\chi$
is a multiplicative function. Finally, from (\ref{elqorachi5}) we
have $$f(x)=\frac{ \int_{S}f(xst)d\upsilon(s)d\upsilon(t)}{
\int_{S}f(t)d\upsilon(t)}=\int_{S}\chi(xs)d\upsilon(s)=\chi(x)\int_{S}\chi(s)d\upsilon(s).$$
This completes the proof.
\end{proof}The  continuous solutions of (\ref{elqorachi2}) are described in
the following theorem.
\begin{thm} Let $S$ be a locally compact  semigroup, let
$\sigma$ : $S\longrightarrow S$ be a continuous involutive
automorphism of $S$, and let  $\upsilon$ be a Borel complex measure
with compact support and which is $\sigma$-invariant. The continuous
solutions of the functional equation (\ref{elqorachi2})
 are the functions
\begin{equation}\label{elqorachi7}
    f=\frac{\psi+\psi\circ\sigma}{2},
\end{equation} where $\psi$: $S\longrightarrow \mathbb{C}$ is a
continuous  $\upsilon$-spherical function. That is
$\int_{S}\psi(xsy)d\upsilon(s)=\psi(x)\psi(y)$ for all $x,y\in S.$
\end{thm}
\begin{proof} Same computations that are needed in our discussion are due to
Stetk\ae r \cite{st4}. Let $f$ be a solution of (\ref{elqorachi2})
and  let $x,y,z\in S$. If we replace $x$ by $xsy$ and $y$ by $z$ in
(\ref{elqorachi2}) and integrating the result obtained with respect
to $s$ we get
\begin{equation}\label{elqorachi8}
\int_{S}\int_{S}f(xsytz)d\upsilon(s)d\upsilon(t)+\int_{S}\int_{S}f(\sigma(z)txsy)d\upsilon(t)d\upsilon(s)
= 2f(z)\int_{S}f(xsy)d\upsilon(s).
\end{equation}
On the other hand if we replace $x$ by $\sigma(z)sx$ in
(\ref{elqorachi2}) and integrate the resul obtained with respect to
$s$ we obtain
$$\int_{S}\int_{S}f(\sigma(z)sxty)d\upsilon(s)d\upsilon(t)+\int_{S}\int_{S}f(\sigma(y)t\sigma(z)sx)d\upsilon(t)d\upsilon(s)=2f(y)\int_{S}f(\sigma(z)sx)d\upsilon(s)$$
$$=2f(y)\bigg[2f(x)f(z)-\int_{S}f(xsz)d\upsilon(s)\bigg].$$
Since, $\upsilon$ is $\sigma$-invariant, so we have
$$\int_{S}\int_{S}f(\sigma(y)t\sigma(z)sx)d\upsilon(t)d\upsilon(s)=\int_{S}\int_{S}f(\sigma(ytz)sx)d\upsilon(t)d\upsilon(s)$$
$$=2f(x)\int_{S}f(ytz)d\upsilon(t)-\int_{S}\int_{S}f(xsytz)d\upsilon(s)d\upsilon(t).$$
Thus, we get
\begin{equation}\label{elqorachi9}
 \int_{S}\int_{S}f(\sigma(z)sxty)d\upsilon(s)d\upsilon(t)+2f(x)\int_{S}f(ysz)d\upsilon(s)-\int_{S}\int_{S}f(xsytz)d\upsilon(s)d\upsilon(t)\end{equation}
 $$=2f(y)\bigg[2f(x)f(z)-\int_{S}f(xsz)d\upsilon(s)\bigg].$$
 Subtracting this from (\ref{elqorachi8}) we get
\begin{equation}\label{elqorachi10}
    \int_{S}\int_{S}f(xsytz)d\upsilon(s)d\upsilon(t)\end{equation}$$=f(x)\int_{S}f(ysz)d\upsilon(s)+f(z)\int_{S}f(xsy)d\upsilon(s)+f(y)\int_{S}f(xsz)d\upsilon(s)-2f(y)f(x)f(z).
$$ With the notation
\begin{equation}\label{elqorachi12}
f_{x}(y)=\int_{S}f(xsy)d\upsilon(s)-f(x)f(y)\end{equation}
 equation (\ref{elqorachi10}) can be
written as follows
\begin{equation}\label{elqorachi13}
\int_{S}f_{a}(xsy)d\upsilon(s)=f_{a}(x)f(y)+f_{a}(y)f(x),\; x,y\in
S,\end{equation} equation which was solved on groups in  \cite{akkouchi}.\\
If $f_{a}=0$ for all $a\in S$, then $f$ is a $\upsilon$-spherical
function. Substituting $f$ into (\ref{elqorachi2}) we obtain that
$f=f\circ\sigma$. So, $f=\frac{\psi+\psi\circ\sigma}{2}$, with
$\psi=f$ a $\upsilon$-spherical function.\\ If there exists $a\in S$
such that  $f_{a}\neq0$ then from \cite{akkouchi},  there exist two
$\upsilon$-spherical functions $\psi_{1},\psi_{2}$:
$S\longrightarrow\mathbb{ C}$ such that $f =
\frac{\psi_1+\psi_2}{2}$. If $\psi_1\neq \psi_2$ by substituting  $f
= \frac{\psi_1+\psi_2}{2}$ into (1.9) we get  after a computation
that
\begin{equation}\label{equation01}
\psi_1(x)[\psi_2(y)-\psi_1(\sigma(y))]+\psi_2(x)[\psi_1(y)-\psi_2(\sigma(y))]=0\end{equation}
for all $x, y\in $. $\psi_1\neq \psi_2$, then  $(\psi_1,\psi_2)$ are
linearly independent \cite{akkouchi}, so from (\ref{equation01}) we
get $\psi_2(y)=\psi_1(\sigma(y))$ for all $y\in S$. This completes
the proof.
\end{proof}By using Theorem 3.2 and Lemma 3.1 we get the following
result.
\begin{cor} Let $S$ be a locally compact semigroup, let
$\sigma$ : $S\longrightarrow S$ be a continuous involutive
automorphism of $S$,  and let $\upsilon$ be a complex Borel
$\sigma$-invariant measure with compact support  contained in the
center of $S$. The continuous solutions of the functional equation
\begin{equation}\label{elqorachi1}
\int_{S}f(xyt)d\upsilon(t)+\int_{S}f(\sigma(y)xt)d\upsilon(t) =
2f(x)f(y), \;x,y\in S\end{equation}
 are the functions
\begin{equation}\label{elqorachi7}
    f=\frac{\chi+\chi\circ\sigma}{2}\int_{S}\chi(t)d\upsilon(t),
\end{equation} where $\chi$: $S\longrightarrow \mathbb{C}$ is a
continuous multiplicative function.
\end{cor}
\begin{rem}Let $S$ be a locally compact semigroup,
$\sigma$ be  a continuous involutive anti-automorphism of $S$, and
$\upsilon$ be a complex measure with compact support and which is
$\sigma$-invariant. If $f$ is a continuous  solution of the
functional equation (\ref{elqorachi2}) then by adapting the
computations used in \cite{wilson5} with the following mappings
$R(y)h(x)=\int_{S}h(xty)d\upsilon(t)$;
$L(y)=\int_{S}f(\sigma(y)tx)d\upsilon(t)$  we get that
$\int_{S}f(xty)d\upsilon(t)=\int_{S}f(ytx)d\upsilon(t)$ for all
$x,y\in S$. So, equation (\ref{elqorachi2}) can be written as follow
$$\int_{S}f(xty)d\upsilon(t)+\int_{S}f(xt\sigma(y))d\upsilon(t)=2f(x)f(y)\;x,y\in
S.$$  The last functional equation has been studied in
\cite{akkouchi}.
\end{rem}
\section{The superstability of the
functional equation (\ref{eq555})}In this section we obtain the
superstability of the variant Van Vleck's functional equation
(\ref{eq555}) on semigroups.
\begin{lem} Let $S$ be a semigroup, let $\sigma$ be an involutive morphism  of $S$. Let $\mu$ be a complex
measure that is a linear combination of Dirac measures
$(\delta_{z_{i}})_{i\in I}$, such that for all $i\in I$, $z_{i}$ is
contained in the center of $S$. Let $\delta>0$ be fixed.  If $f:
S\longrightarrow \mathbb{C}$ is an unbounded function which
satisfies the inequality
\begin{equation}\label{hindd1}
 \bigg|\int_{S}f(\sigma(y)xt)d\mu(t)-\int_{S}f(xyt)d\mu(t)-2f(x)f(y)\bigg|\leq \delta
\end{equation} for all $x,y\in S$. Then, for all $x,y\in S$
\begin{equation}\label{hind2}
    f(\sigma(x))=-f(x),
\end{equation}
\begin{equation}\label{salmahind}
 |f(\sigma(x)y)+f(\sigma(y)x)|\leq
 \frac{3\delta\|\mu\|}{|\int_{S}f(s)d\mu(s)|},
\end{equation}
\begin{equation}\label{hind3}
 \bigg|\int_{S}\int_{S}f(x\sigma(s)t)d\mu(s)d\mu(t)-f(x)\int_{S}f(s)d\mu(s)\bigg|\leq\frac{\delta\|\mu\|}{2},
\end{equation}
\begin{equation}\label{hind4}
\bigg|\int_{S}\int_{S}f(xst)d\mu(s)d\mu(t)+f(x)\int_{S}f(t)d\mu(t)\bigg|\leq
\frac{3\delta\|\mu\|}{2},
\end{equation}
\begin{equation}\label{hind5}
\int_{S}f(t)d\mu(t)\neq 0,
\end{equation}
\begin{equation}\label{hind6}
 \int_{S}f(x\sigma(s)d\mu(s)=\int_{S}f(\sigma(x)\sigma(s))d\mu(s),\end{equation}
 \begin{equation}\label{hind7}
 \bigg|\int_{S}f(xs)d\mu(s)-\int_{S}f(\sigma(x)s)d\mu(s)\bigg|\leq
 \frac{6\delta\|\mu\|^{2}}{|\int_{S}f(s)d\mu(s)|}.
\end{equation}
The function $g$ defined by
\begin{equation}\label{hind8}g(x)=
\frac{\int_{S}f(xt)d\mu(t)}{\int_{S}f(t)d\mu(t)}\;\text{for}\;x\in
S\end{equation} is unbounded on $S$ and satisfies the following
inequality
\begin{equation}\label{hind9}
 |g(xy)+g(\sigma(y)x)-2g(x)g(y)|\leq
 \frac{3\delta\|\mu\|^{2}}{(|\int_{S}f(s)d\mu(s))^{2}|}.
\end{equation} for all $x,y\in S$. Furthermore, $g$ satisfies (1.8),
$\int_{S}\int_{S}g(st)d\mu(s)d\mu(t)\neq0$ and
$\int_{S}g(s)d\mu(s)=0$.
\end{lem}
\begin{proof} Equation (4.6): Let $f$:  $S\longrightarrow \mathbb{C}$ be an unbounded function
which satisfies (4.1).  Equation (4.6): First, We prove that
$\int_{S}f(s)d\mu(s)\neq 0$. Assume that $\int_{S}f(s)d\mu(s)=0$. By
replacing $y$ by $s$ and $x$ by $\sigma(y)x$ in (4.1) and
integrating the result obtained with respect to $s$ we get
\begin{equation}\label{ah10}
    \bigg|\int_{S}\int_{S}f(\sigma(y)x\sigma(s)t)d\mu(s)d\mu(t)-\int_{S}\int_{S}f(\sigma(y)xst)d\mu(s)d\mu(t)\bigg|\leq\delta\|\mu\|.
\end{equation}Replacing  $y$ by $ys$ in (4.1) and integrating the result obtained
with respect to $s$ we have \begin{equation}\label{ah11}
    \bigg|\int_{S}\int_{S}f(\sigma(y)x\sigma(s)t)d\mu(s)d\mu(t)-\int_{S}\int_{S}f(xyst)d\mu(s)d\mu(t)-2f(x)\int_{S}f(ys)d\mu(s)\bigg|\leq\delta\|\mu\|.\end{equation}
On the other hand by replacing  $x$ by $xs$ in (4.1) and integrating
the result obtained with respect to $s$ we obtain
\begin{equation}\label{ah12}
    \bigg|\int_{S}\int_{S}f(\sigma(y)xst)d\mu(s)d\mu(t)-\int_{S}\int_{S}f(xyst)d\mu(s)d\mu(t)-2f(y)\int_{S}f(xs)d\mu(s)\bigg|\leq\delta\|\mu\|.\end{equation}
    By subtracting  the result of equation (\ref{ah12}) from the result of (\ref{ah11}) and
    using the triangle inequality, we get after computation that
    \begin{equation}\label{ah13}
   \bigg |\int_{S}\int_{S}f(\sigma(y)x\sigma(s)t)d\mu(s)d\mu(t)-\int_{S}\int_{S}f(\sigma(y)xst)d\mu(s)d\mu(t)\end{equation}
    $$-2[f(x)\int_{S}f(ys)d\mu(s)-f(y)\int_{S}f(xs)d\mu(s)]\bigg|\leq2\delta\|\mu\|.$$
From  (\ref{ah10}), (\ref{ah13}) and
    the triangle inequality we get
    \begin{equation}\label{ah14}
    \bigg|f(x)\int_{S}f(ys)d\mu(s)-f(y)\int_{S}f(xs)d\mu(s)\bigg|\leq\frac{3\delta}{2}\|\mu\|.\end{equation}
    $f$ is assumed to be  unbounded function on $S$ then $f\neq 0$. Let $y_{0}\in
    S$ such that $f(y_0)\neq 0$. Equation (\ref{ah14}) can be
    written as follows \begin{equation}\label{ah15}
    \bigg|\int_{S}f(xs)d\mu(s)-\alpha f(x)|\leq\frac{3\delta}{2|f(y_0)\bigg|}\|\mu\|,\end{equation}
    where $\alpha=\frac{\int_{S}f(y_{0}s)d\mu(s)}{f(y_0)}$. Of
    course $\alpha\neq 0$ because if $\alpha=0$ then by using
    (\ref{ah15}) we deduce that the function $x\longmapsto \int_{S}f(xs)d\mu(s)$ is bounded and from (4.1) and the triangle inequality we get  $f$ bounded. Which contradict the
    assumption that $f$ is an unbounded function on $S$.\\
    Now, from (\ref{ah15}) and the triangle inequality  equation (4.1) can be written as
    follows \begin{equation}\label{ah16}
 |f(\sigma(y)x)-f(xy)-\frac{2}{\alpha}f(x)f(y)|\leq
 \frac{\frac{3\delta\|\mu\|}{|f(y_0)|}+\delta}{|\alpha|}=M
\end{equation} for all $x,y\in S$. Since $\int_{S}f(s)d\mu(s)=0$. So, if we replace $y$ by  $s$ in (\ref{ah16}) and integrating the result obtained with respect to $s$ we
get \begin{equation}\label{ah17}
 \bigg|\int_{S}f(\sigma(s)x)d\mu(s)-\int_{S}f(xs)d\mu(s)\bigg|\leq M\|\mu\|\;\text{for all}\; x\in S.
\end{equation}
Replacing  $y$ by $x$ and $x$ by $s$ in (\ref{ah16}) and integrating
the result obtained with respect to $s$ we get
\begin{equation}\label{ah18}
 \bigg|\int_{S}f(\sigma(x)s)d\mu(s)-\int_{S}f(sx)d\mu(s)\bigg|\leq M\|\mu\|\;\text{for all}\;x\in S.
\end{equation}Subtracting the result of (\ref{ah17}) from the result of (\ref{ah18}) and
using the triangle inequality we get \begin{equation}\label{ah19}
 \bigg|\int_{S}f(\sigma(x)s)d\mu(s)-\int_{S}f(\sigma(s)x)d\mu(s)\bigg|\leq 2M\|\mu\|\;\text{for all}\;x\in S.
\end{equation}By interchanging $x$ with $y$ in (\ref{ah16}) we get
\begin{equation}\label{moh1}
 |f(\sigma(x)y)-f(yx)-\frac{2}{\alpha}f(x)f(y)|\leq
 M.
\end{equation} If we replace $y$ by $\sigma(y)$ in (\ref{ah16}) we
have \begin{equation}\label{moh2}
 |f(yx)-f(x\sigma(y))-\frac{2}{\alpha}f(x)f(\sigma(y))|\leq
 M.
\end{equation} By adding  the results of (\ref{moh2}) and (\ref{moh1}) and using the triangle inequality we obtain
\begin{equation}\label{moh3}
 |f(\sigma(x)y)-f(x\sigma(y))-\frac{2}{\alpha}f(x)[f(\sigma(y))+f(\sigma(y))]|\leq
 2M.
\end{equation} By replacing $x$ by $\sigma(x)$ in (\ref{moh3}) we
get \begin{equation}\label{moh4}
 |f(xy)-f(\sigma(x)\sigma(y))-\frac{2}{\alpha}f(\sigma(x))[f(\sigma(y))+f(\sigma(y))]|\leq
 2M.
\end{equation} If we replace $y$ by $\sigma(y)$ in (\ref{moh3}) we
get \begin{equation}\label{moh5}
 |f(\sigma(x)\sigma(y))-f(xy)-\frac{2}{\alpha}f(x)[f(\sigma(y))+f(\sigma(y))]|\leq
 2M.
\end{equation}Now, by adding the results of (\ref{moh4}) and
(\ref{moh5}) and using the triangle inequality we have
\begin{equation}\label{moh6}
 |[f(x)+f(\sigma(x))][f(\sigma(y))+f(\sigma(y))]|\leq
 2M|\alpha|.
\end{equation} That is $x\longmapsto f(x)+f(\sigma(x))$ is a bounded
function on $S$. So, the function $x\longmapsto
\int_{S}f(\sigma(x)s)d\mu(s)+\int_{S}f(x\sigma(s))d\mu(s)$ is also a
bounded function on $S$. Since, from (\ref{ah19}) we have
$x\longmapsto
\int_{S}f(\sigma(x)s)d\mu(s)-\int_{S}f(x\sigma(s))d\mu(s)$ is a
bounded function on $S$. Consequently, the function $x\longmapsto
\int_{S}f(xs)d\mu(s)$ is a bounded function on $S$ and from (4.1)
and the triangle inequality we get that  $f$ is a bounded function
on $S$. Which contradict the assumption that $f$ is an unbounded
function on $S$ and  this proves (\ref{hind5}).\\Equation (4.3): By
replacing $y$ by $ys$ in (4.1) and integrating the result obtained
with respect to $s$ we get
\begin{equation}\label{ah29}
 \bigg|\int_{S}\int_{S}f(\sigma(y)x\sigma(s)t)d\mu(s)d\mu(t)-\int_{S}\int_{S}f(xyst)d\mu(s)d\mu(t)-2f(x)\int_{S}f(ys)d\mu(s)\bigg|\leq
 \delta\|\mu\|.
\end{equation}If we replace $x$ by $xs$ in (4.1) and integrating the result obtained with respect to $s$ we get
\begin{equation}\label{ah30}
 \bigg|\int_{S}\int_{S}f(\sigma(y)xst)d\mu(s)d\mu(t)-\int_{S}\int_{S}f(xyst)d\mu(s)d\mu(t)-2f(y)\int_{S}f(xs)d\mu(s)\bigg|\leq
 \delta\|\mu\|.
\end{equation} By subtracting the result of (\ref{ah30}) from the result of (\ref{ah29}) and
using the triangle inequality we obtain
\begin{equation}\label{ah31}
 \bigg|\int_{S}\int_{S}f(\sigma(y)x\sigma(s)t)d\mu(s)d\mu(t)-\int_{S}\int_{S}f(\sigma(y)xst)d\mu(s)d\mu(t)\end{equation}
 $$-2[f(x)\int_{S}f(ys)d\mu(s)-f(y)\int_{S}f(xs)d\mu(s)]\bigg|\leq 2\delta\|\mu\|.
$$Replacing $y$ by $s$ and $x$ by $\sigma(y)x$ in (4.1) and
integrating the result obtained with respect to $s$ we get
\begin{equation}\label{ah32}
 \bigg|\int_{S}\int_{S}f(\sigma(y)x\sigma(s)t)d\mu(s)d\mu(t)-\int_{S}\int_{S}f(\sigma(y)xst)d\mu(s)d\mu(t)-2f(\sigma(y)x)\int_{S}f(s)d\mu(s)\bigg|\leq \delta\|\mu\|.
\end{equation} By subtracting the result of (\ref{ah31}) from the result of (\ref{ah32}) and
using the triangle inequality we obtain
\begin{equation}\label{ah33}
 |f(\sigma(y)x)\int_{S}f(s)d\mu(s)-[f(x)\int_{S}f(ys)d\mu(s)-f(y)\int_{S}f(xs)d\mu(s)]|\leq \frac{3\delta\|\mu\|}{2}.
\end{equation}By interchanging $x$ and $y$ in (\ref{ah33}) we have
\begin{equation}\label{ah34}
 |f(\sigma(x)y)\int_{S}f(s)d\mu(s)-[f(y)\int_{S}f(xs)d\mu(s)-f(x)\int_{S}f(ys)d\mu(s)]|\leq \frac{3\delta\|\mu\|}{2}.
\end{equation}By adding the result of (\ref{ah33}) and the result of (\ref{ah34}) and using the triangle inequality we get
\begin{equation}\label{ah35}
 |f(\sigma(x)y)+f(\sigma(y)x)|\leq \frac{3\delta\|\mu\|}{|\int_{S}f(s)d\mu(s)|}.
\end{equation}for all $x,y\in S$. This proves
(\ref{salmahind}).\\Equation (4.7):  Replacing $x$ by $x\sigma(s)$
in (4.1) and integrating the result obtained with respect to $s$ we
get\begin{equation}\label{ah36}
 \bigg|\int_{S}\int_{S}f(\sigma(y)x\sigma(s)t)d\mu(s)d\mu(t)-\int_{S}\int_{S}f(xy\sigma(s)t)d\mu(s)d\mu(t)-2f(y)\int_{S}f(x\sigma(s))d\mu(s)\bigg|\leq
 \delta\|\mu\|.
\end{equation}If we replace $y$ by $ys$ and $x$ by $xt$ in (\ref{ah35}) and
integrating the result obtained with respect to $s$ and $t$ we
obtain
\begin{equation}\label{ah37}
 \bigg|\int_{S}\int_{S}f(\sigma(x)y\sigma(t)s)d\mu(s)d\mu(t)+\int_{S}\int_{S}f(\sigma(y)x\sigma(s)t)d\mu(s)d\mu(t)\bigg|\leq \frac{3\delta\|\mu\|^{3}}{|\int_{S}f(s)d\mu(s)|}.
\end{equation} By subtracting the result of (\ref{ah36}) from  the
result of (\ref{ah37}) and using the triangle inequality we get
\begin{equation}\label{ah38}
\bigg
|\int_{S}\int_{S}f(\sigma(x)y\sigma(t)s)d\mu(s)d\mu(t)+\int_{S}f(xy\sigma(s)t)d\mu(s)d\mu(t)+2f(y)\int_{S}f(x\sigma(s)d\mu(s)\bigg|
\end{equation}$$\leq
  \frac{3\delta\|\mu\|^{3}}{|\int_{S}f(s)d\mu(s)|}+\delta\|\mu\|.$$ Replacing $x$ by $\sigma(x)$ in (\ref{ah38}) we get
 \begin{equation}\label{ah39}
 \bigg|\int_{S}\int_{S}f(xy\sigma(t)s)d\mu(s)d\mu(t)+\int_{S}f(\sigma(x)y\sigma(s)t)d\mu(s)d\mu(t)+2f(y)\int_{S}f(\sigma(x)\sigma(s)d\mu(s)\bigg|
\end{equation}$$\leq
  \frac{3\delta\|\mu\|^{3}}{|\int_{S}f(s)d\mu(s)|}+\delta\|\mu\|.$$ Subtracting the result of (\ref{ah39}) from  the
result of (\ref{ah38}) and using the triangle inequality we get
\begin{equation}\label{ah40}
 \bigg|2f(y)[\int_{S}f(x\sigma(s)d\mu(s)-\int_{S}f(\sigma(x)\sigma(s)d\mu(s)]\bigg|
\leq
  \frac{6\delta\|\mu\|^{3}}{|\int_{S}f(s)d\mu(s)|}+2\delta\|\mu\|.\end{equation} Since $f$ is assumed to be unbounded then we
 deduce (\ref{hind6}).\\Equation (4.8): If we replace $y$ by $\sigma(s)$ in (\ref{ah35}) and
integrating the result obtained with respect to $s$  we obtain
\begin{equation}\label{ah41}
 \bigg|\int_{S}f(xs)d\mu(s)+\int_{S}f(\sigma(x)\sigma(s))d\mu(s)\bigg|\leq \frac{3\delta\|\mu\|^{2}}{|\int_{S}f(s)d\mu(s)|}.
\end{equation}In view of (\ref{hind6}) the inequality (\ref{ah41}) can be written as follows
 \begin{equation}\label{ah42}
 \bigg|\int_{S}f(xs)d\mu(s)+\int_{S}f(x\sigma(s))d\mu(s)\bigg|\leq \frac{3\delta\|\mu\|^{2}}{|\int_{S}f(s)d\mu(s)|}.
\end{equation}Replacing $y$ by $s$ in (\ref{ah35}) and
integrating the result obtained with respect to $s$  we get
\begin{equation}\label{ah43}
 \bigg|\int_{S}f(x\sigma(s))d\mu(s)+\int_{S}f(\sigma(x)s)d\mu(s)\bigg|\leq \frac{3\delta\|\mu\|^{2}}{|\int_{S}f(s)d\mu(s)|}.
\end{equation}By subtracting the result of (\ref{ah42}) from the
result of (\ref{ah43}) we obtain (\ref{hind7}).\\
Equation (4.2): Replacing $x$ by $\sigma(x)$ in (4.1) we get
\begin{equation}\label{ah44}
\bigg
|\int_{S}f(\sigma(y)\sigma(x)t)d\mu(t)-\int_{S}f(\sigma(x)yt)d\mu(s)d\mu(t)-2f(\sigma(x))f(y)\bigg|\leq
\delta.
\end{equation}Now, we will discuss two cases.
\\Case 1. If $\sigma$ is an involutive automorphism of $S$. Replacing $x$ by $yx$ in (\ref{hind7}) we
obtain
\begin{equation}\label{ah45}
\bigg
|\int_{S}f(yxs)d\mu(s)-\int_{S}f(\sigma(y)\sigma(x)s)d\mu(s)\bigg|\leq
 \frac{6\delta\|\mu\|^{2}}{|\int_{S}f(s)d\mu(s)|}.
\end{equation}Adding the result of (\ref{ah44}) to the result of (\ref{ah45}) and using
the triangle inequality we get
\begin{equation}\label{ah46}
 \bigg|\int_{S}f(yxs)d\mu(s)-\int_{S}f(\sigma(x)ys)d\mu(s)-2f(\sigma(x))f(y)\bigg|\leq
 \frac{6\delta\|\mu\|^{2}}{|\int_{S}f(s)d\mu(s)|}+\delta.
\end{equation}
By interchanging $x$ and $y$ in (4.1) we get
\begin{equation}\label{ah47}
\bigg
|\int_{S}f(\sigma(x)ys)d\mu(s)-\int_{S}f(yxs)d\mu(s)-2f(x)f(y)\bigg|\leq\delta.
\end{equation} By adding the result of (\ref{ah47}) and the result of (\ref{ah46}) and using
the triangle inequality we obtain \begin{equation}\label{ah48}
 |2f(y)[f(\sigma(x))+f(x)]|\leq
 \frac{6\delta\|\mu\|^{2}}{|\int_{S}f(s)d\mu(s)|}+2\delta.
\end{equation}Since $f$ is unbounded then we get (\ref{hind2}).\\
Case 2. If $\sigma$ is an involutive anti-automorphism of $S$. By
replacing  $x$ by $yx$ in (\ref{hind7}) we have
\begin{equation}\label{aj1}
 \bigg|\int_{S}f(yxs)d\mu(s)-\int_{S}f(\sigma(x)\sigma(y)s)d\mu(s)\bigg|\leq
 \frac{6\delta\|\mu\|^{2}}{|\int_{S}f(s)d\mu(s)|}.
\end{equation}If we replace $y$ by $x$ and $y$ by $\sigma(y)$ in
(4.1) we get \begin{equation}\label{aj2}
 |\int_{S}f(\sigma(x)\sigma(y)t)d\mu(t)-\int_{S}f(\sigma(y)xt)d\mu(t)-2f(\sigma(y))f(x)|\leq\delta.
\end{equation} By adding the results of (\ref{aj1}) and  (\ref{aj2})
and using the triangle inequality we get \begin{equation}\label{aj3}
 \bigg|\int_{S}f(yxt)d\mu(t)-\int_{S}f(\sigma(y)xt)d\mu(t)-2f(\sigma(y))f(x)\bigg|\leq\delta+\frac{6\delta\|\mu\|^{2}}{|\int_{S}f(s)d\mu(s)}|.
\end{equation} From (4.1) we have
\begin{equation}\label{aj4}
 \bigg|\int_{S}f(\sigma(x)yt)d\mu(t)-\int_{S}f(yxt)d\mu(t)-2f(y)f(x)\bigg|\leq\delta.
\end{equation} and from (\ref{hind7}) we have
\begin{equation}\label{aj5}
\bigg
|\int_{S}f(\sigma(x)ys)d\mu(s)-\int_{S}f(\sigma(y)xs)d\mu(s)\bigg|\leq
 \frac{6\delta\|\mu\|^{2}}{|\int_{S}f(s)d\mu(s)|}\end{equation}By
 subtracting  the results of (\ref{aj4}) and (\ref{aj5}) and using the triangle inequality we get
 \begin{equation}\label{aj6}
 \bigg|\int_{S}f(\sigma(y)xt)d\mu(t)-\int_{S}f(yxt)d\mu(t)-2f(y)f(x)\bigg|\leq\delta+\frac{6\delta\|\mu\|^{2}}{|\int_{S}f(s)d\mu(s)|}.
\end{equation}By adding the results of (\ref{aj6}) and (\ref{aj3})
 \begin{equation}\label{aj7}|2f(x)(f(y)+f(\sigma(y)))|\leq\frac{6\delta\|\mu\|^{2}}{|\int_{S}f(s)d\mu(s)|}+\delta
\end{equation} for all $x,y\in S.$ Since $f$ is unbounded then we we get (\ref{hind2}).\\
Equation (4.4): If we replace $x$ by $\sigma(s)$ in (4.1) and using
(4.2) and integrating the result obtained with respect to $s$ we get
\begin{equation}\label{ah49}
\bigg
|-\int_{S}\int_{S}f(ys\sigma(t))d\mu(s)d\mu(t)-\int_{S}\int_{S}f(y\sigma(s)t)d\mu(s)d\mu(t)+2f(y)\int_{S}f(s)d\mu(s)\bigg|\leq\delta\|\mu\|,
\end{equation}which proves (\ref{hind3})\\ Equation (4.5): By replacing $y$ by $s$ in
(4.1) and integrating the result obtained with respect to $s$ we get
\begin{equation}\label{ah50}
 \bigg|\int_{S}\int_{S}f(x\sigma(s)t)d\mu(s)d\mu(t)-\int_{S}\int_{S}f(xst)d\mu(s)d\mu(t)-2f(x)\int_{S}f(s)d\mu(s)\bigg|\leq\delta\|\mu\|.
\end{equation} From (\ref{hind3}) and the triangle inequality we
obtain $$
 \bigg|\int_{S}\int_{S}f(xst)d\mu(s)d\mu(t)+f(x)\int_{S}f(s)d\mu(s)\bigg|\leq\delta\|\mu\|+\frac{\delta\|\mu\|}{2}.$$  for all $x\in S.$ This proves (\ref{hind4}).
 \\Equation (4.10): Let $g$ be the function defined by
$g(x)=
\frac{\int_{S}f(xt)d\mu(t)}{\int_{S}f(t)d\mu(t)}\;\text{for}\;x\in
S.$ Then we have
 \begin{equation}\label{ah51}
\int_{S}f(s)d\mu(s)\int_{S}f(k)d\mu(k)[g(xy)+g(\sigma(y)x)-2g(x)g(y)]
 \end{equation}
 $$=\int_{S}f(k)d\mu(k)\int_{S}f(xyt)d\mu(t)+\int_{S}f(s)d\mu(s)\int_{S}f(\sigma(y)xt)d\mu(t)-2\int_{S}f(xs)d\mu(s)\int_{S}f(ys)d\mu(s)$$
 $$=\int_{S}[f(xyt)\int_{S}f(k)d\mu(k)+\int_{S}\int_{S}f(xytks)d\mu(k)d\mu(s)]d\mu(t)$$
 $$+\int_{S}[f(\sigma(y)xt)\int_{S}f(s)d\mu(s)-\int_{S}\int_{S}f(\sigma(y)xt\sigma(k)s)d\mu(k)d\mu(s)]d\mu(t)$$
 $$+\int_{S}\int_{S}[\int_{S}f(\sigma(yk)xst)d\mu(t)-\int_{S}f(xsykst)d\mu(t)$$
 $$-2f(xs)f(yk)]d\mu(k)d\mu(s).$$ So, from (\ref{hind3}), (\ref{hind4}) and (4.1) we get
 $$\int_{S}f(s)d\mu(s)\int_{S}f(k)d\mu(k)[g(xy)+g(\sigma(y)x)-2g(x)g(y)]$$
 $$\leq\frac{3\delta\|\mu\|^{2}}{2}+\frac{\delta\|\mu\|^{2}}{2}+\delta\|\mu\|^{2}
 =3\delta\|\mu\|^{2}.$$ Which proves (\ref{hind9}). Now, since $f$ is unbounded then $g$ is unbounded and satisfies (\ref{hind9}). So,   by using  same computations
used in \cite{{Elq-Red}} $g$ satisfies the variant of d'Alembert's
functional equation $g(xy)+g(\sigma(y)x)=2g(x)g(y)$ for all $x,y\in S.$\\
Finally,  From (\ref{hind3}), (\ref{hind4}) and the triangle
inequality we have
\begin{equation}\label{sar1}
\bigg|\int_{S}\int_{S}f(x\sigma(s)t)d\mu(s)d\mu(t)+\int_{S}f(xst)d\mu(s)d\mu(t)\bigg|\leq
2\delta\|\mu\|,
\end{equation} for all $x,y\in S.$ By using the definition of $g$ the
inequality (\ref{sar1}) can be written as follows
$$\bigg|\int_{S}f(k)d\mu(k)[\int_{S}g(x\sigma(k))d\mu(k)+\int_{S}g(xk)d\mu(k)]\bigg|\leq2\delta\|\mu\|.$$
On the other hand $g$ is a solution of d'Alembert's functional
equation (1.8) then $g$ is central and  we get
$|2g(x)\int_{S}g(k)d\mu(k)|\leq\frac{2\delta\|\mu\|}{|\int_{S}f(k)d\mu(k)|}$
for all $x\in S.$ Since $g$ is unbounded then we deduce that
$\int_{S}g(k)d\mu(k)=0$. That is
$\int_{S}\int_{S}f(st)d\mu(s)d\mu(t)=0$.
\end{proof}
\begin{thm} Let $S$ be a semigroup, let $\sigma$ be an involutive morphism  of $S$. Let $\mu$ be a complex
measure that is a linear combination of Dirac measures
$(\delta_{z_{i}})_{i\in I}$, such that for all $i\in I$, $z_{i}$ is
contained in the center of $S$. Let $\delta>0$ be fixed.  If $f:
S\longrightarrow \mathbb{C}$ is a function which satisfies the
inequality
\begin{equation}\label{hindd1}
 \bigg|\int_{S}f(\sigma(y)xt)d\mu(t)-\int_{S}f(xyt)d\mu(t)-2f(x)f(y)\bigg|\leq \delta
\end{equation} for all $x,y\in S$. Then, either $f$ is bounded on $S$ and $|f(x)|\leq\frac{\|\mu\|+\sqrt{\|\mu\|^{2}+2\delta}}{2}$
for all $x\in S$ or $f$ is a solution of the variant Van Vleck's
functional equation (\ref{eq555}).
\end{thm} \begin{proof} Assume that $f$ is an unbounded solution of (\ref{hindd1}).  Replacing $y$ by  $ys$ in (4.1)
and integrating the result obtained with respect to $s$ we get
\begin{equation}\label{om1}
 \bigg|\int_{S}\int_{S}f(\sigma(y)x\sigma(s)t)d\mu(s)d\mu(t)-\int_{S}\int_{S}f(xyst)d\mu(s)d\mu(t)-2f(x)\int_{S}f(ys)d\mu(s)\bigg|\leq
 \delta \|\mu\|
\end{equation} for all $x,y\in S$. By using    (4.1), (4.4) and (4.5)  and the triangle
inequality we get
\begin{equation}\label{om2}
 \bigg|\int_{S}f(s)d\mu(s)f(xy)+\int_{S}f(s)d\mu(s)f(\sigma(y)x)-2f(x)\int_{S}f(ys)d\mu(s)\bigg|\leq
 {3\delta \|\mu\|}
\end{equation} for all $x,y\in S$. Equation which can be written as follows
\begin{equation}\label{om2}
 | f(xy)+ f(\sigma(y)x)-2f(x)g(y)|\leq
 \frac{3\delta \|\mu\|}{|\int_{S}f(s)d\mu(s)|}
\end{equation} for all $x,y\in S$ and where $g$ is the function defined above.
Replacing $x$ by  $xs$ in (4.1) and integrating the result obtained
with respect to $s$ we get
\begin{equation}\label{oma1}
 \bigg|\int_{S}\int_{S}f(\sigma(y)xst)d\mu(s)d\mu(t)-\int_{S}\int_{S}f(xyst)d\mu(s)d\mu(t)-2f(y)\int_{S}f(xs)d\mu(s)\bigg|\leq
 \delta \|\mu\|
\end{equation} for all $x,y\in S$. By using    (4.1), (4.5)   and the triangle
inequality we get
\begin{equation}\label{oma2}
 | f(xy)-f(\sigma(y)x)-2f(y)g(x)|\leq
 \frac{4\delta \|\mu\|}{|\int_{S}f(s)d\mu(s)|}
\end{equation} for all $x,y\in S$. By adding the result of
(\ref{oma2}) and (\ref{om2}) we get
\begin{equation}\label{oma3}
 | f(xy)-f(x)g(y)-f(x)g(y)|\leq
 \frac{2\delta \|\mu\|}{|\int_{S}f(s)d\mu(s)|}
\end{equation} for all $x,y\in S$. Now, we will show  that if  $\alpha f+\beta
g$ is a bounded function on $S$ then $\alpha=\beta=0$.  Assume that
there exits $M$ such that
\begin{equation}\label{hamid1}
 |\alpha f(x)+\beta \int_{S}f(xt)d\mu(t)|\leq M\end{equation}
  for all $x\in S.$ Then $|\alpha f(\sigma(x))+\beta
\int_{S}f(\sigma(x)t)d\mu(t)|\leq M$. Since $f(\sigma(x))=-f(x)$.
So, by using (4.8) and the triangle inequality we get
\begin{equation}\label{hamid2}
|-\alpha f(x)+\beta\int_{S}f(xt)d\mu(t)|\leq
M+|\beta|\frac{6\delta\|\mu\|^{2}}{|\int_{S}f(t)d\mu(t)|}.\end{equation}
By adding the result of (\ref{hamid1}) and (\ref{hamid2}) we get
$2\beta\int_{S}f(xt)d\mu(t)$ is a bounded function. Since $g$ is
unbounded then $\beta=0$ and consequently $\alpha=0$. Now, from
[\cite{z4}, Lemma 2.1]  we conclude that $f,g$ are solutions of the
sine addition law
\begin{equation}\label{om3}
    f(xy)=f(x)g(y)+f(y)g(x)\;x,y\in S.
\end{equation} Since $f(\sigma(x))=-f(x)$ and $g(\sigma(x)=g(x)$
for all $x\in S$ then the pair $f,g$ satisfies the variant  Wilson's
functional equation \begin{equation}\label{omaa3}
    f(xy)+f(\sigma(y)x)=2f(x)g(y)\;x,y\in S.
\end{equation}Taking $y=s$ in (\ref{omaa3}) and integrating the result
obtained with respect to $s$ we get
\begin{equation}\label{om4}
\int_{S}f(xs)d\mu(s)+\int_{S}f(\sigma(s)x)d\mu(s)=0,
\end{equation} because $\int_{S}g(s)d\mu(s)=0$. By replacing $y$ by
$s\sigma(k)$ in in (\ref{om3}) and integrating the result obtained
with respect to $s$ and $k$ we obtain
$$
\int_{S}\int_{S}f(xs\sigma(k))d\mu(s)d\mu(k)+\int_{S}\int_{S}f(xs\sigma(k))d\mu(s)d\mu(k)=2f(x)\int_{S}\int_{S}g(s\sigma(k))d\mu(s)d\mu(k).
$$
That is \begin{equation}\label{om4}
\int_{S}f(xs\sigma(k))d\mu(s)d\mu(k)=f(x)\int_{S}\int_{S}g(s\sigma(k))d\mu(s)d\mu(k).
\end{equation} Now from (4.4) and (\ref{om4}) we get $$|f(x)(\int_{S}\int_{S}g(s\sigma(k))d\mu(s)d\mu(k)-
\int_{S}f(t)d\mu(t))|\leq\frac{\delta\|\mu\|}{2}$$ for all $x\in S.$
Since $f$ is assumed to be unbounded then we get
\begin{equation}\label{om5}
\int_{S}\int_{S}g(s\sigma(k))d\mu(s)d\mu(k)=\int_{S}f(t)d\mu(t).
\end{equation} $g$ satisfies (1.8) and  $\int_{S}g(s)d\mu(s)=0$ implies that $\int_{S}g(yk)d\mu(s)=-\int_{S}g(y\sigma(k))d\mu(s)$. So,  by using the definition of $g$,
equations (\ref{om4}) and (\ref{om5})  we have
\begin{equation}\label{2019}
\int_{S}g(yk)d\mu(k)=-\int_{S}g(y\sigma(k))d\mu(k)=\frac{-\int_{S}\int_{S}f(y\sigma(k)t)d\mu(k)d\mu(t)}{\int_{S}f(s)d\mu(s)}\end{equation}$$
=\frac{-f(y)\int_{S}\int_{S}g(\sigma(k)t)d\mu(k)d\mu(t)}{\int_{S}f(s)d\mu(s)}=\frac{-f(y)\int_{S}f(t)d\mu(t)}{\int_{S}f(s)d\mu(s)}=-f(y).
$$ Finally, From  (\ref{om3}), (4.63), (\ref{om4}) and (\ref{2019}) for all $x,y\in
S$ we have
$$\int_{S}f(\sigma(y)xt)d\mu(t)-\int_{S}f(xyt)d\mu(t)=-\int_{S}f(\sigma(y)x\sigma(t))d\mu(t)-\int_{S}f(xyt)d\mu(t)$$
$$=-[\int_{S}f(\sigma(yt)x)d\mu(t)+\int_{S}f(xyt)d\mu(t)]$$
$$=-2f(x)\int_{S}g(yt)d\mu(t)=2f(x)f(y).$$ That is $f$ is a solution
of Van Vleck's functional equation (\ref{eq555}). This completes the
proof.

\end{proof}


\vspace{1cm} Authors addresses\\
Elqorachi Elhoucien, Redouani Ahmed\\Ibn Zohr University, Faculty of
Sciences\\Department of Mathematic,  Agadir, Morocco\\ E-mail:
{elqorachi@hotmail.com;
Redouani$_{-}$ahmed@yahoo.fr}\\\\
Themistocles M. Rassias \\Department of Mathematics, National
Technical University of Athens, Zografou Campus, 15780, Athens
Greece,\\E-mail: {trassias@math.ntua.gr}

\end{document}